\documentclass[12pt]{amsart}
\usepackage{amsmath,amssymb,amsthm}
\usepackage{enumerate}

\newtheorem{thm}{Theorem}
\newtheorem{lem}[thm]{Lemma}

\newtheorem{prop}[thm]{Proposition}

\newtheorem{defn}[thm]{Definition}

\usepackage[colorinlistoftodos]{todonotes}

\renewcommand{\Diamond}{\mc{D}_2}
\newcommand{\sat}{{\rm sat}}
\newcommand{\F}{\mathcal{F}}
\newcommand{\ids}{induced-$\Diamond$-saturated }
\newcommand{\Bn}{\mathcal{B}_n}
\newcommand{\La}{{\rm La}}
\newcommand{\B}{{\mathcal B}}
\newcommand{\mc}{\mathcal}

\newenvironment{pfct}[1]{\noindent{\bf Proof of #1.\,}}{\hfill$\Box$ \\}

\begin{document}
\title[Improved bounds for induced poset saturation]{Improved bounds for induced poset saturation}
\author{Ryan R. Martin}
\address{(Martin) Department of Mathematics, Iowa State University, Ames, Iowa, 50011}
\email{rymartin@iastate.edu}
\author{Heather C. Smith}
\address{(Smith) Department of Mathematics and Computer Science, Davidson College, Davidson, NC 28035}
\email{hcsmith@davidson.edu}
\author{Shanise Walker}
\address{(Walker) Department of Mathematics, University of Wisconsin-Eau Claire, Eau Claire, WI 54702 }
\email{walkersg@uwec.edu}
\maketitle

\begin{abstract}
	Given a finite poset $\mathcal{P}$, a family $\mathcal{F}$ of elements in the Boolean lattice is induced-$\mathcal{P}$-saturated if $\mathcal{F}$ contains no copy of $\mathcal{P}$ as an induced subposet but every proper superset of $\mathcal{F}$ contains a copy of $\mathcal{P}$ as an induced subposet.  The minimum size of an induced-$\mathcal{P}$-saturated family in the $n$-dimensional Boolean lattice, denoted $\sat^*(n,\mathcal{P})$, was first studied by Ferrara et al. (2017).
	
	Our work focuses on strengthening lower bounds. For the 4-point poset known as the diamond, we prove $\sat^*(n,\Diamond)\geq\sqrt{n}$, improving upon a logarithmic lower bound. For the antichain with $k+1$ elements, we prove $\sat^*(n,\mathcal{A}_{k+1})\geq (1-o_k(1))\frac{kn}{\log_2 k}$, improving upon a lower bound of $3n-1$ for $k\geq 3$. 
\end{abstract}

\section{Introduction}\label{introposets}

A \emph{partially ordered set}, or \emph{poset}, $\mc{P}=(P,\preceq)$ consists of a set $P$ with a binary relation ``$\preceq$'' which is reflexive, antisymmetric, and transitive. For $x,y\in P$, we say $x$ and $y$ are \emph{comparable} provided $x\preceq y$ or $y\preceq x$. Otherwise, we say $x$ and $y$ are \emph{incomparable}, writing $x\parallel y$.  For posets, $\mc{P}=(P, \preceq)$ and $\mc{P}^{\prime}=(P^{\prime}, \preceq)$, we say $\mc{P}^{\prime}$ is a  \emph{(weak) subposet} of $\mc{P}$ if there exists an injection $f:{P}^{\prime}\rightarrow {P}$ that preserves the partial ordering, i.e. whenever $u\preceq^{\prime} v$, we have $f(u)\preceq f(v)$. The \emph{$n$-dimensional Boolean lattice}, $\B_n=(2^{[n]}, \subseteq)$ is the poset consisting of all subsets of $[n]:=\{1,\ldots, n\}$, ordered by inclusion.  

For a poset $\mc{Q}=(Q,\preceq)$ and a family $\F \subseteq Q$, we often refer to $\F=(\F,\preceq)$ as the \emph{poset induced by $\F$ in $\mc{Q}$} which is obtained by restricting the binary relation $\preceq$ to $\F\times \F$. For any poset $\mc{P}$, we say that $\F$ is \emph{$\mc{P}$-free} provided $\mc{P}$ is not a subposet of $\F$. The maximum size of a $\mc{P}$-free family in $\B_n$ is denoted $\La(n,\mc{P})$. 

The study of $\mathcal P$-free families dates back to Sperner's~\cite{Sperner} proof that the maximum size of an antichain (a poset in which no two distinct elements are comparable) in $\B_n$ is ${\binom{n}{\lfloor n/2\rfloor}}$. In particular, this is the quantity $\La(n,\mc{P}_2)$ where $\mc{P}_2 = (\{x,y\}, \leq)$ such that $x\leq y$.  See Griggs and Li~\cite{GriggsLi} for an extensive survey of $\mc{P}$-free results. 

Extremal questions about sizes of $\mc{P}$-free families  are only interesting as maximization problems, however the notion of saturation allows us to formulate corresponding minimization questions. Before defining saturation for posets, we first give some background on the much older problem of graph saturation.

Let $G$ and $H$ be graphs. A spanning subgraph $F$ of $G$ is \emph{$(G,H)$-saturated} if $F$ contains no copy of $H$ as a subgraph but, for any edge $e\in E(G)- E(F)$, the graph $F+e=(V(F),E(F) \cup \{e\})$  contains a copy of $H$ as a subgraph. Let $\sat(G,H)$ denote the minimum number of edges in a $(G,H)$-saturated graph.

Saturation was introduced by Zykov \cite{Zyk} and  Erd\H{o}s, Hajnal, and Moon \cite{EHM} independently, when they considered $\sat(K_n,H)$, for $K_n$ being the complete graph on $n$ vertices. Erd\H{o}s, Hajnal, and Moon showed that for $n\geq k\geq 2$, $\sat(K_n, K_k)=(k-2)n-\binom{k-1}{2}$. Since its introduction, saturation has been studied extensively. 
Faudree, Faudree, and Schmitt~\cite{Faudree} provide a dynamic survey of saturated graphs.  

Saturation for posets was first introduced by Gerbner, et al.~\cite{Gerbner}.
\begin{defn}Given a host poset $\mc{Q}=(Q, \leq_{\mc{Q}})$, a target poset $\mc{P}=(P,\leq_{\mc{P}})$, and a family $\mc{F}\subseteq  Q$, we  say that $\mc{F}$ is \emph{$\mc{P}$-saturated} in $\mc{Q}$ if  the following  two properties hold:
\begin{itemize}
\item the poset induced by $\mc{F}$ in $\mc{Q}$ is $\mc{P}$-free, and 
\item for any $S\in Q-\mc{F}$, the poset $\mc{P}$ is a subposet of the poset induced by $\mc{F}\cup \{S\}$ in $\mc{Q}$.
\end{itemize}
\end{defn}

Given $n>0$ and a poset $\mc{P}$, the \emph{saturation number}, denoted $\sat(n, \mc{P})$, is the minimum size of a $\mc{P}$-saturated family in $\B_n$.     Let $\mc{P}_k$ denote the \emph{chain} consisting of $k$ elements in which each pair is comparable.  Gerbner et al. proved the following: 
 
\begin{thm}[Gerbner et al.~\cite{Gerbner}]\label{thm:Gerbner} For $n$ sufficiently large, 
\[2^{k/2-1}\leq \sat(n, \mc{P}_{k+1})\leq 2^{k-1}.\]
\end{thm}

Morrison, Noel, and Scott~\cite{Morrison} improved the upper bound when they found a smaller construction for a $\mc{P}_7$-saturated family and used an iterative process to obtain the following upper bound:

\begin{thm}[Morrison, Noel, Scott~\cite{Morrison}]\label{thm:MNSupper}
There exists $\varepsilon>0$ such that for all $k>0$ and $n$ sufficiently large, 
\[\sat(n,\mc{P}_{k+1})\leq 2^{(1-\varepsilon)k}.\] In particular, $\epsilon = \left(1-\frac{\log_2 15}{4} \right)\approx 0.023277$ suffices.
\end{thm}

Ferrara et al.~\cite{Ferrara} found that the saturation number $\sat(n,\mc{P})$ was constant in $n$ for many posets $\mc{P}$, so they introduced the notion of induced-saturation.

A poset $\mc{P}' = (P',\preceq')$ is an \emph{induced-subposet} of $\mc{P}=(P,\preceq)$ if there exists an injective function $f: {P}'\rightarrow {P}$ such that $u\preceq ' v$ if and only if $f(u)\preceq f(v)$.

\begin{defn} [Ferrara et al.~\cite{Ferrara}] 
 Given a host poset $\mc{Q}=(Q, \leq_{\mc{Q}})$, a target poset $\mc{P}=(P,\leq_{\mc{P}})$, and a family $\mc{F}\subseteq  Q$, we  say that $\mc{F}$ is \emph{induced-$\mc{P}$-saturated} in $\mc{Q}$ if  the following  two properties hold:
\begin{itemize}
\item the poset induced by $\F$ in $\mc{Q}$ does not contain an induced copy of $\mc{P}$, and 
\item for any $S\in {Q}-\mc{F}$,  the poset induced by $\mc{F}\cup \{S\}$ in $\mc{Q}$ contains an induced copy of $\mc{P}$.
\end{itemize}
\end{defn}

Given $n>0$ and a poset $\mc{P}$, the \emph{induced saturation number},  denoted $\sat^*(n, \mc{P})$, is the  {minimum} size of a family $\F$ that is induced-$\mc{P}$-saturated in $\B_n$.

To summarize some of the results of Ferrara et al.~\cite{Ferrara}, we first define several posets.  Let $\mc{A}_k$ denote the antichain with $k$ elements. The Hasse diagrams for $\mc{V}_2$, the butterfly $\bowtie$, and the diamond $\mc{D}_2$ are given in Figure~\ref{fig:posets}.

\begin{figure}[ht]
\begin{center}
\begin{tabular}{ccc}
\begin{tikzpicture}
\foreach \x in {(-.5,1),(0,0),(.5,1)} \draw [fill=black] \x circle (0.06);
\draw (-.5,1)--(0,0)--(.5,1);
\node at (0,-.5) {$\mc{V}_2$};
\end{tikzpicture}
& \hspace{.6in}
\begin{tikzpicture}
\foreach \x in {(-.5,0),(.5,0),(-.5,1), (.5,1)} \draw [fill=black] \x circle (0.06);
\draw (-.5,0)--(-.5,1)--(.5,0)--(.5,1)--(-.5,0);
\node at (0,-.5) {\text{the butterfly, }$\bowtie$};
\end{tikzpicture}
& \hspace{.6in}
\begin{tikzpicture}
\foreach \x in {(-.3,.5),(0,0),(.3,.5), (0,1)} \draw [fill=black] \x circle (0.06);
\draw (-.3,.5)--(0,0)--(.3,.5)--(0,1)--(-.3,.5);
\node at (0,-.5) {\text{the diamond, }$\mc{D}_2$};

\end{tikzpicture}

\end{tabular}
\end{center}
\caption{Hasse diagrams for three named posets.}
\label{fig:posets}
\end{figure}
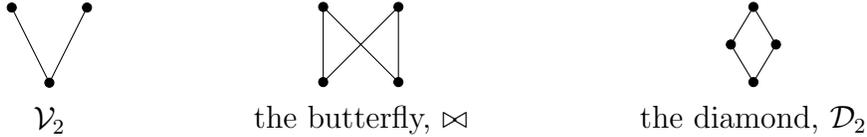 

Ferrara et al. proved the following regarding induced-$\mc{P}$-saturation for particular posets. 

\begin{thm}[Ferrara et al.~\cite{Ferrara}]\label{thm:inducedsatF} 
~\smallskip
\begin{enumerate}[(a)]
\item If $n\geq 2$, then $\sat^*(n, \mc{V}_2)=n+1$.
\item If $n\geq 3$, then $\lceil\log_2 n\rceil\leq \sat^*(n, \bowtie)\leq \binom{n}{2}+2n-1$.
\item If $n\geq 2$, then $\lceil\log_2 n\rceil\leq \sat^*(n, \Diamond)\leq n+1$. \label{it:diam}
\item If $n>k\geq 3$, then 
	\begin{equation*}
		3n-1\leq \sat^*(n, \mc{A}_{k+1})\leq (n-1)k-\left(\frac{1}{2}\log_2 k+\frac{1}{2}\log_2\log_2 k+O(1)\right).
	\end{equation*} \label{it:anti}
\end{enumerate}
\end{thm}

In this paper, we improve lower bounds for induced saturation numbers for the $k$-antichain and the diamond. In particular, Theorem~\ref{thm:diamond} improves the lower bound in Theorem~\ref{thm:inducedsatF}{\it (\ref{it:diam})} and Theorem~\ref{thm:antichain} improves the lower bound in  Theorem~\ref{thm:inducedsatF}{\it (\ref{it:anti})}.
 
 \begin{thm}\label{thm:diamond}
For $n>0$, $\sat^*(n,\Diamond)\geq\lceil\sqrt{n}\,\rceil$.
\end{thm}
 
 \begin{thm}\label{thm:antichain}
For $k \geq 3$ and $n$ sufficiently large,
$\sat^*(n, \mc{A}_{k+1})\geq \left(1-\frac{1}{\log_2 k}\right)\frac{kn}{\log_2 k}$.
\end{thm} 
We note at the end of Section~\ref{antichainresults} that $n\geq\log_2^3 k$ suffices to satisfy the bound in Theorem~\ref{thm:antichain}. 

In Section~\ref{diamondresults} we prove Theorem~\ref{thm:diamond} while Section~\ref{antichainresults} contains the proof of Theorem~\ref{thm:antichain}.


\section{Induced-$\Diamond$-saturated results}\label{diamondresults}

Theorem~\ref{thm:inducedsatF}{\it(\ref{it:diam})} gives a lower bound of $\lceil\log_2 n\rceil$ for $\sat^*(n,\Diamond)$, which is far from the upper bound of $n+1$. 
As presented in \cite{Ferrara}, this logarithmic lower bound holds for $\sat^*(n,\mc{P})$ for a much larger class of posets,  however they did obtain the following partial result for $\sat^*(n,\mathcal{D}_2)$   which is specific to the diamond: 

\begin{thm}[Ferrara et al.~\cite{Ferrara}]
\label{thm:emptyfull}
Let $\F$ be an induced-$\Diamond$-saturated family in $\B_n$. If $\emptyset \in \F$ or $[n]\in \F$, then $|\F|\geq n+1$. 
\end{thm}
\noindent This led to the conjecture $\sat^*(n,\Diamond) = \Theta(n)$  \cite{Ferrara}.  

We prove that any induced-$\Diamond$-saturated family must have size at least $\sqrt{n}$. To establish this, we first need the following technical lemma. 

\begin{lem}\label{lem:diamondkey}
	For $n\geq 1$, let $\F$ be an \ids~family in $\Bn$ and let $F^*$ be an arbitrary member of $\F$. 
		\begin{enumerate}
		\item	For each $i \in F^*$, there exist $F_i,G_i \in \F$ such that $F_i \subseteq F^*$ and  $F_i-G_i = \{i\}$.
				 \label{it:sub}
		\item	For each $j \not\in F^*$, there exist $G_j,F_j \in \F$ such that $G_j \supseteq F^*$ and $F_j-G_j = \{j\}$.\label{it:sup}
	\end{enumerate}
\end{lem}

\begin{proof}
	We will prove {\it(\ref{it:sub})} and the proof for  {\it(\ref{it:sup})} follows similarly. 
	
	For {\it(\ref{it:sub})}, first observe that if $F^*=\emptyset$, then the statement is trivially true. Assume $F^* \neq \emptyset$ and fix $i \in F^*$.

	If there is $F \in \F$ such that $\{i\}\subseteq F\subseteq F^*$ and $F-\{i\} \in \F$, then $F_i=F$ and $G_i=F-\{i\}$ have the desired properties to satisfy the lemma.

	Now suppose no such pair exists in $\F$. Choose $F$ to be a set of minimum size in $\{F\in \F:\{i\}\subseteq F \subseteq F^*\}$. Since $F-\{i\} \not\in \F$ and $\F$ is an \ids family, the family $\F \cup \{F-\{i\}\}$ must contain an induced copy of a diamond. This diamond will have $F-\{i\}$ as a top element, a bottom element, or a side element. We examine each possibility separately.
	
	In the case where $F-\{i\}$ forms the top element of a diamond, there exist distinct $S,T_1,T_2\in {\mathcal F}$ such that $T_1$ and $T_2$ are incomparable while $S\subset T_j \subset F-\{i\}$ for each $j\in\{1,2\}$. Thus, $F\in{\mathcal F}$ forms the top element of the diamond $\{S,T_1,T_2,F\}$, a contradiction to $\mathcal F$ being diamond-free.
		
	In the case where $F-\{i\}$ forms the bottom element of the diamond, there exist distinct $T_1,T_2,U \in \F$ such that $T_1$ and $T_2$ are incomparable while $F-\{i\} \subset T_j \subset U$ for each $j\in \{1,2\}$.
	Since $F, T_1, T_2, U$ does not form a diamond, we conclude $i\not\in T_j$ for some $j\in \{1,2\}$.
	Without loss of generality, $F-\{i\}\subset T_1$ but $F\not\subset T_1$.
	So, we can set $F_i = F$ and $G_i = T_1$ to satisfy the conclusion of the lemma in this case.
	
	Finally, in the case where $F-\{i\}$ forms the side element of a diamond, there exist distinct $S,T,U\in {\mathcal F}$ such that $T$ and $F-\{i\}$ are incomparable while $S\subset T\subset U$ and $S\subset F-\{i\} \subset U$. Now we consider two subcases based on the relationship of $F$ to $U$.
		
	If $F \not\subseteq U$, then we can set $F_i = F$ and $G_i = U$ since $F-\{i\}\subset U$ while $F \not\subset U$.
	
	Otherwise $F \subseteq U$. Since $\F$ is induced-$\mc{D}_2$-saturated, $S,F,T,U$ does not form a diamond, so $F$ and $T$ must be comparable. 
	Since $F-\{i\}$ and $T$ are incomparable, we must have $T\subset F$ and thus $\{i\} \subseteq T \subset F \subseteq F^*$.
	But this contradicts the choice of $F$ as a minimum-sized subset of $F^*$ in $\F$ that contains $i$.
\end{proof}

Now we have the tools to prove our main result for induced-$\Diamond$-saturation.
\smallskip

\begin{pfct}{Theorem~\ref{thm:diamond}}
	 Construct a simple directed graph $D$ as follows: The vertex set will be $\F$ and we include arc $(A,B)$ if and only if $|B - A|=1$. Thus $D$ has at most one arc for each ordered pair of vertices.
	
	On the other hand, Lemma~\ref{lem:diamondkey} gives a lower bound on the number arcs in our digraph. Fix any $F^{*}\in\F$.
 
	 For each $i\in F^{*}$ there is an arc $(F_i,G_i)$, and for each $j\in [n]-F^{*}$ there is an arc $(F_j,G_j)$. Consequently, there are at least $n$ distinct arcs in $D$.  
	
	Because the total number of arcs in this graph is at most $|{\mathcal F}|(|{\mathcal F}|-1)$, it is the case that $|{\mathcal F}|(|{\mathcal F}|-1)\geq n$ and so $|{\mathcal F}|\geq \lceil\sqrt{n}\rceil$.
	\end{pfct}

We can obtain a much better lower bound when the \ids  family exhibits certain properties. Let $\F$ be an \ids family. If $\emptyset\in \F$ or $[n]\in \F$, then Theorem~\ref{thm:emptyfull} gives a lower bound of $n+1$ for $|\F|$. On the other hand, Proposition~\ref{prop:hilo} shows that $|\F|$ is also large whenever $\F$ contains a large minimal element or a small maximal element. 

\begin{prop}\label{prop:hilo}
	Let $\F$ be an \ids~family in $\Bn$. 
	\begin{enumerate}
		\item 	If $S^*$ is minimal in $\F$, then $|S^*|+1\leq |\F|$. \label{it:sstar}
		\item 	If $U^*$ is maximal in $\F$, then $n-|U^*|+1\leq |\F|$. \label{it:ustar}
	\end{enumerate}
\end{prop}

\begin{proof}
	For~{\it(\ref{it:sstar})}, the result is trivial if $\emptyset \in \F$, so assume otherwise and fix a non-empty, minimal $S^*$ in $\F$ and $i \in S^*$.  
	Because $S^*$ is minimal in $\F$, we know $S^* - \{i\}\not\in \F$ and $\F \cup \{S^* - \{i\}\}$ must contain an induced diamond. In particular, there exist distinct $U,T_1,T_2 \in \F$ such that $T_1$ and $T_2$ are incomparable while $ S^* - \{i\} \subset T_j \subset U$ for each $j\in \{1,2\}$.
	
	Since $U,T_1,T_2,S^*\in\F$ do not induce a diamond, $i\not\in T_j$ for some $j\in \{1,2\}$. Without loss of generality, $i\not\in T_1$, and thus $S^* - T_1 = \{i\}$.  
	In particular, for each $i \in S^*$, there exists a unique $T(i)\in \F$ such that $S^*-T(i) = \{i\}$.
	These sets in $\F$, together with $S^*$, imply $|\F| \geq |S^*|+1$.

Part~{\it(\ref{it:ustar})} follows from duality and Part~{\it(\ref{it:sstar})}.
\end{proof}


\section{Induced-$\mc{A}_{k+1}$-saturated results}\label{antichainresults}
 
Now we turn our attention to antichains. The values $\sat^*(n, \mc{A}_{2})=n+1$ and $\sat^*(n, \mc{A}_{3})=2n$ are proven in~\cite{Ferrara}. While the first is clear, the lower bound for the second requires a non-trivial proof. Theorem~\ref{thm:inducedsatF}{\it (\ref{it:anti})} gives that $3n-1 \leq {\rm sat}^*(n,{\mathcal A}_4)\leq 3n+O(1)$. For $k\geq 4$, Theorem~\ref{thm:inducedsatF}{\it (\ref{it:anti})} gives a lower bound of ${\rm sat}^*(n,{\mathcal A}_k)\geq 3n-1$ which is far from the upper bound of approximately $kn$.
 
To establish $\sat^*(n, \mc{A}_{k+1})\geq (1-o_k(1))\frac{kn}{\log_2 k}$, we prove a more technical statement in Theorem~\ref{thm:satlowerbd} below, using two different counting arguments, and Theorem~\ref{thm:antichain} will follow as a corollary. While two bounds are proven, at the end of this section a discussion is included about which of them is best for fixed $k$ and sufficiently large $n$.  
 
\begin{thm}\label{thm:satlowerbd}
Let $k\geq 3$ be an integer and let $\mc{A}_{k+1}$ be an antichain of size ${k+1}$. Then for all $n\geq k$,
\begin{enumerate}
\item $\sat^*(n, \mc{A}_{k+1})\geq k\left\lceil\frac{n}{\lfloor\log_2 k\rfloor +1}\right\rceil -k+2$.\label{it:a}
\item $\sat^*(n, \mc{A}_{k+1})\geq 2n+ \sum\limits_{j=3}^{k}\left\lceil\frac{n}{d^*(j)}\right\rceil -k+2$ where $d^*(j)$ is the largest $d$ such that $\binom{d}{\lfloor d/2\rfloor}\leq j-1$. \label{it:b}
\end{enumerate}
\end{thm}

\begin{proof}
We will use the following classical theorem of Dilworth:



\begin{thm}[Dilworth~\cite{Dilworth}]\label{thm:Dilworth}
Let $\mc{P}=(P, \leq)$ be a (finite) poset. If $k $ is the size of the largest antichain in $\mc{P}$,  then there exists a family of disjoint chains $\{\mc{C}_1, \mc{C}_2, \dots, \mc{C}_k\}$  such that ${P}=\mc{C}_1\cup \mc{C}_2\cup \dots \cup \mc{C}_k$. 
\end{thm}

Let $\F$ be an induced-$\mc{A}_{k+1}$-saturated family in $\B_n$.  Because adding either $\emptyset$ or $[n]$ to a nonempty family will not increase the size of the largest antichain, $\F$ must contain both $\emptyset$ and $[n]$.  In addition, the largest antichain in $\F-\{\emptyset, [n]\}$ has size $k$. By Dilworth's theorem,  $\F-\{\emptyset, [n]\}$ can be partitioned into $k$ disjoint chains $\mc{C}_1, \mc{C}_2, \dots, \mc{C}_k$. There may be many such partitions, we choose an arbitrary one. We then add both $\emptyset$ and $[n]$ to each of the chains. Therefore, for all distinct $i,j$ in $[k]$, it is the case that $\mc{C}_i\cap \mc{C}_j = \{\emptyset, [n]\}$.

For $i\in [k]$ and $X,Y\in \mc{C}_i$, we call the open interval $(X,Y)=\{Z\in 2^{[n]}: X\subset Z\subset Y\}$ a \emph{gap} provided there is no $Z\in \mathcal{C}_i$ with $X\subset Z \subset Y$. We define the \emph{size} of the gap $(X,Y)$ to be $|Y-X|$. 

For all $i\in [k]$, the following observations will be needed in both parts of the proof: 
\begin{enumerate}[(i)]
\item If $d$ is the maximum size of a gap in $\mc{C}_i$, then the number of elements in $\mc{C}_i$ is at least $\lceil n/d\rceil +1$. \label{it:1}
\item If $(X,Y)$ is a gap in $\mc{C}_i$, then every $Z\in (X,Y)$ must be in $\F$. \label{it:2}
\end{enumerate}
Item \eqref{it:1} follows from the observation that $\{\emptyset,[n]\}\subset \mc{C}_i$ and item \eqref{it:2} follows from the fact that if this were not the case, then $\F\cup\{Z\}$ can be partitioned into $k$ chains, which means that the largest antichain has size $k$, a contradiction to induced saturation. 

\bigskip
\noindent {\bf Part~{\it(\ref{it:a})}:}
 For any $i\in [k]$, we say a gap $(X,Y)$ in $\mc{C}_i$ is \emph{wide} if there exists $j\neq i$ such that $|\mc{C}_j\cap (X,Y)| \geq 3$.

In the case where there is at least one wide gap, let $\bar{d}$ be the maximum size among all wide gaps which is realized by gap $(X,Y)$ in  $\mc{C}_i$ with $S, T, U\in \mc{C}_j\cap (X,Y)$ for $j\neq i$. Without loss of generality, assume $(S,T)$ and $(T,U)$ are gaps in $\mc{C}_j$. 
By removing $T$ from $\mc{C}_j$ and assigning it to $\mc{C}_i$, we obtain a new set of $k$ chains that partition $\F-\{\emptyset, [n]\}$ such that the maximum size among wide gaps has not increased. In particular, the gap $(X,Y)$ has been replaced with gaps $(X,T)$ and $(T,Y)$, each with size at most $\bar{d}-2$. On the other hand, gaps $(S,T)$ and $(T,U)$ have been replaced with $(S,U)$ which has size at most $\bar{d}-2$ since $X\subset S \subset U \subset Y$.

Further, note that any gap in some $\mc{C}_{i'}$ containing $X,T,Y$ must also contain $S,T,U$. 
So since $\bar{d}$ was the maximum size of a wide gap, the number of wide gaps with size $\bar{d}$ has decreased by at least one.
Consequently, performing this operation on wide gaps of maximum size, one at a time, the process eventually terminates with a chain partition of $\F-\{\emptyset, [n]\}$ in which no gaps are wide.

Now we may assume there are no wide gaps in our $k$ chains, so each gap will contain at most two elements from each other chain. Recall from item \eqref{it:2} that each element of $\B_n$ in a gap must be in $\F$, so if $d$ is the maximum gap size, then we have  $2(k-1)\geq 2^{d}-2$. This implies $\log_2 k +1\geq {d}$, which gives 
\begin{equation*}
	\sat^*(n, \mc{A}_{k+1})\geq k\left(\left\lceil \frac{n}{{d}}\right\rceil -1\right)+2 \geq k\left\lceil \frac{n}{\lfloor\log_2 k\rfloor +1}\right\rceil -k+2,
\end{equation*}
completing the proof of part {\it(\ref{it:a})}.  
                    
\bigskip
\noindent {\bf Part~{\it(\ref{it:b})}:} Return to the original chain partition $\mc{C}_1, \ldots, \mc{C}_k$ of $\F-\{\emptyset,[n]\}$ guaranteed by Dilworth's Theorem. Now we use an recursive process to obtain a coloring of $\F-\{\emptyset,[n]\}$ with $k$ colors. To begin, all elements of $\F-\{\emptyset,[n]\}$ are considered to be uncolored. 
Build a maximal chain in $\F-\{\emptyset,[n]\}$ by starting with $\mc{C}_1$ and add elements of $\F-\{\emptyset,[n]\}$, one at a time, provided the result is still a chain. Continue until no more elements of $\F-\{\emptyset,[n]\}$ can be added. The elements in the final chain will form color class 1. Now for some $1\leq j< k$, suppose that we have already defined color classes $1, \ldots, j$. To obtain color class $j+1$, start with $\mc{C}_{j+1}$, remove any elements that have already been colored from $\F-\{\emptyset,[n]\}$ and add uncolored elements from $\F-\{\emptyset,[n]\}$, one at a time, provided the result is still a chain. Continue until this process terminates. The elements in the resulting chain will form color class $j+1$. Finally add $\emptyset$ and $[n]$ to every color class.

By item \eqref{it:2}, each element of $\B_n$ in a gap must be in $\F$, thus the structure of our color classes guarantees that each element in $(X,Y)$  has a color from the set $\{1,\ldots, j-1\}$. Since each color class forms a chain, every element in the largest antichain in $(X,Y)$ must have a different color. For a gap $(X,Y)$ which has size $d$, note that $(X,Y) \cup \{X,Y\}$ is isomorphic to $\B_d$.  Sperner's theorem~\cite{Sperner} gives  that the maximum size of an antichain in $\B_d$ is $\binom{d}{\lfloor d/2\rfloor}$, therefore $\binom{d}{\lfloor d/2\rfloor}\leq j-1$.

For each $j\in \mathbb{N}$, let  $d^*(j)$ be the  largest integer $d$ for which $\binom{d}{\lfloor d/2\rfloor}\leq j-1$. Therefore, each gap in the chain for color class $j$ has size at most $d^*(j)$ and by item \eqref{it:1}, color class $j$ has at least $\lceil n/d^*(j)\rceil +1$ elements. Correcting for $\emptyset$ and $[n]$ being in every color class and noting that $d^*(1)=d^*(2)=1$, we obtain the desired bound:
\begin{equation*}
\sat^*(n, \mc{A}_{k+1}) \geq  \sum\limits_{j=1}^{k}\left(\left\lceil\frac{n}{d^*(j)}\right\rceil -1\right)+2= 2n+ \sum\limits_{j=3}^{k}\left\lceil\frac{n}{d^*(j)}\right\rceil-k+2.
\end{equation*}
\end{proof}

A table of values of $k$, $1\leq k\leq 300$, is given in~\cite{Walker} which determines whether Theorem~\ref{thm:satlowerbd}{\it(\ref{it:a})} or Theorem~\ref{thm:satlowerbd}{\it(\ref{it:b})} is the better lower bound for $\sat^*(n, \mc{A}_{k+1})$ when $n$ is sufficiently large. 

For $n$ sufficiently large and $k\leq 243$ the lower bound provided by {\it(\ref{it:b})} is larger than the bound provided by {\it(\ref{it:a})}. However, for large $k$ ($k\geq 2^{64}$ suffices), {\it(\ref{it:a})} is larger than {\it(\ref{it:b})}. Some straightforward calculations show that {\it(\ref{it:a})} is at least $\left(1-\frac{1}{\log_2k}\right)\frac{kn}{\log_2k}$ for $n\geq\log_2^3k$ and that suffices for our main result. For further discussion of the comparison between {\it(\ref{it:a})} and {\it(\ref{it:b})} see \cite[Section 3.2]{Walker}. 

\section{Conclusion}

Despite the significant improvement over the results in~\cite{Ferrara}, we still believe that $\sat^*(n,\Diamond)=n+1$ and $\sat^*(n, \mc{A}_{k+1})=kn-o(k)$. 

For the case of $\sat^*(n, \mc{A}_{k+1})$ for $k\geq 3$, we note that Theorem~\ref{thm:inducedsatF}{\it(\ref{it:anti})} shows that $\sat^*(n, \mc{A}_{k+1})<(n-1)k-(1/2+o(1))\log_2k$, so a lower bound of $(k-1)n$ is not possible.

As seen in Theorem~\ref{thm:inducedsatF}, there are still many open questions with induced saturation and other poset families to be explored.

Martin was partially supported by a grant from the Simons Foundation (\#353292). Smith acknowledges partial support from NSF-DMS grant \#1344199.  Most of the work for this project was completed while Walker was affiliated with Iowa State University.


\begin{thebibliography}{10}

\bibitem{Dilworth}
R.~P. Dilworth.
\newblock A decomposition theorem for partially ordered sets.
\newblock {\em Annals of Mathematics}, {\bf 51}(1): 161--166, 1950.

\bibitem{EHM}
P.~Erd{\"o}s, A.~Hajnal, and J.~W. Moon.
\newblock A problem in graph theory.
\newblock {\em The American Mathematical Monthly}, {\bf 71}(10): 1107--1110, 1964.

\bibitem{Faudree}
J.~R. Faudree, R.~J. Faudree, and J.~R. Schmitt.
\newblock A survey of minimum saturated graphs.
\newblock {\em The Electronic Journal of Combinatorics}, Dynamic Survey, DS19, accessed 14 March 2019.

\bibitem{Ferrara}
M.~Ferrara, B.~Kay, L.~Kramer, R.~R. Martin, B.~Reiniger, H.~C. Smith, and
  E.~Sullivan.
\newblock The saturation number of induced subposets of the Boolean lattice.
\newblock {\em Discrete Mathematics}, {\bf 340}(10): 2479--2487, 2017.

\bibitem{Gerbner}
D.~Gerbner, B.~Keszegh, N.~Lemons, C.~Palmer, D.~P{\'a}lv{\"o}lgyi, and
  B.~Patk{\'o}s.
\newblock Saturating {S}perner {F}amilies.
\newblock {\em Graphs and Combinatorics}, {\bf 29}(5): 1355--1364, 2013.

\bibitem{GriggsLi}
J.~R. Griggs and W.-T. Li.
\newblock Progress on poset-free families of subsets.
\newblock In {\em Recent Trends in Combinatorics}, pages 317--338. Springer,
  2016.

\bibitem{Morrison}
N.~Morrison, J.~A. Noel, and A.~Scott.
\newblock On saturated $k$-{S}perner systems.
\newblock {\em The Electronic Journal of Combinatorics}, {\bf 21}(3): P3--22, 2014.

\bibitem{Sperner}
E.~Sperner.
\newblock Ein satz {\"u}ber untermengen einer endlichen menge.
\newblock {\em Mathematische Zeitschrift}, {\bf 27}(1): 544--548, 1928.

\bibitem{Walker}
S.~Walker. 
\newblock Problems in extremal graphs and poset theory
\newblock (2018). Graduate Theses and Dissertations. 16482. 

\bibitem{Zyk}
A.~A. Zykov.
\newblock On some properties of linear complexes.
\newblock {\em Matematicheskii {S}bornik}, {\bf 66}(2): 163--188, 1949.

\end{thebibliography}


\end{document}